\documentclass[11pt]{amsart}
\usepackage{amssymb,amsmath}
\usepackage[left=3cm,right=3cm,top=3cm,bottom=3cm]{geometry}
\usepackage{hyperref}

\newtheorem{thm}{Theorem}[section]
\newtheorem{prop}[thm]{Proposition}
\newtheorem{lem}[thm]{Lemma}

\theoremstyle{definition}

\newtheorem{rem}[thm]{Remark}
\newtheorem{ex}[thm]{Example}

\DeclareMathOperator{\Ad}{Ad}
\DeclareMathOperator{\ad}{ad}
\DeclareMathOperator{\dist}{dist}
\DeclareMathOperator{\tr}{tr}
\newcommand{\SU}{\mathrm{SU}}
\newcommand{\su}{\mathfrak{su}}
\newcommand{\U}{\mathrm{U}}
\newcommand{\uu}{\mathfrak{u}}
\newcommand{\A}{\mathcal{A}}
\newcommand{\C}{\mathbb{C}}
\newcommand{\R}{\mathbb{R}}
\newcommand{\nn}[1]{\|#1\|}

\title[Short curves in homogeneous spaces of the unitary group]{Short curves in homogeneous spaces of the unitary group\\ for unitarily invariant norms}
\author{Gabriel Larotonda}
\address{Instituto Argentino de Matem\'atica (CONICET) and Facultad de Ciencias Exactas y Naturales, Universidad de Buenos Aires, Argentina}
\email{glaroton@dm.uba.ar}


\begin{document}
\begin{abstract}
Let $\A\subset M_n(\C)$ be a unital $*$-subalgebra and let $M=\U(n)/\U(\A)$, with the quotient Finsler metric induced by a unitarily invariant norm on $\uu(n)$, normalized so that $\nn{iP}=1$ for rank-one projections $P$. This family contains the unitary orbits of Hermitian matrices, in particular flag manifolds and Grassmann manifolds. We prove that if $z_0$ is a minimal lifting of a tangent vector $X$ with $\nn{X}_o<\frac{\pi}{2}$, then the one-parameter group $t\mapsto e^{tz_0}\cdot o$ has minimal length among all piecewise smooth curves in $M$ joining its endpoints. The same holds in $\SU(n)/(\SU(n)\cap\A)$ for every $\Ad$-invariant Finsler norm on $\su(n)$, not necessarily symmetric. The proof relies on the convexity of the distance from the identity along geodesic segments, valid for all these norms. We also prove that the ball of radius $\frac{\pi}{2}$ around the identity, for the spectral norm, is convex along one-parameter segments. This extends known results for the $p$-Schatten norms with $p$ even, and gives a new proof for the operator norm, where the result is due to Dur\'an, Mata-Lorenzo and Recht.
\end{abstract}
\maketitle

\section{Introduction}

Short curves in homogeneous spaces of the unitary group have been studied in operator theory at least since the work of Porta and Recht on Grassmann manifolds in 1987 \cite{pr87}; the systematic study of homogeneous spaces of the unitary group of a C$^*$-algebra, with quotient Finsler metrics, is due to Dur\'an, Mata-Lorenzo and Recht \cite{dmlr04,dmlr05}. Let us recall the setting. If $M=G/K$, the length of a tangent vector $X$ is the distance, in the Lie algebra of $G$, from any lifting of $X$ to the Lie algebra of $K$; a lifting that realizes this distance is called \emph{minimal}. The natural question is whether the one-parameter group generated by a minimal lifting projects onto a short curve of $M$. For Riemannian metrics (or smooth enough Finsler metrics) this follows from the theory of isometric submersions. For the norms that appear in operator theory, which are neither smooth nor strictly convex, one is faced with the lack of such tools; the question is stated in full generality in \cite[Remark 3.31]{lar19}.

For the operator norm, Dur\'an, Mata-Lorenzo and Recht proved that if $Z$ is a minimal lifting, then $e^{tZ}\cdot o$ is short for $|t|\le\frac{\pi}{2\nn{Z}}$ \cite[Theorem I-2]{dmlr04}. Minimal liftings exist for W$^*$-algebras, hence in particular for every $*$-subalgebra of $M_n(\C)$, and in finite dimension this settles the case of the operator norm.

A different way to obtain minimality in the quotient is by means of \emph{convexity in the group}. A minimal lifting $z_0$ is a critical point of the distance from the identity, restricted to the fibre $e^{z_0}K$; if this distance is convex along the geodesic segments of the fibre, the critical point is a minimum. This argument was used in \cite[Theorem 4.9]{alr10} for the $p$-Schatten norms with $p$ even, where convexity comes from a trace computation, and in \cite[Theorem 5.3]{al10} for the operator norm, by letting $p\to\infty$; this gives a second proof of the result of \cite{dmlr04}, for orbits of finite-rank operators. Both techniques depend on the particular norm, and none of them applies to other unitarily invariant norms, such as the Ky Fan norms or the $p$-Schatten norms with $p$ odd.

In joint work with M.~Miglioli \cite{lm26} we proved that for every unitarily invariant norm $\nn{\cdot}$ on $\uu(n)$ and every geodesic segment $t\mapsto e^{tx}e^{y}$ of $\U(n)$, the map $t\mapsto\nn{\log(e^{tx}e^{y})}$ is convex as long as the segment stays in the neighbourhood $\{u:\nn{u-1}_\infty<\sqrt2\}$, and that this neighbourhood is optimal; the same holds in $\SU(n)$ for every $\Ad$-invariant Finsler norm on $\su(n)$, not necessarily symmetric. In this note we show that this convexity, together with a convexity property of the spectral ball of radius $\frac{\pi}{2}$ (Lemma \ref{ball}), is enough to run the argument above for every unitarily invariant norm (for the operator norm this gives a new proof of \cite[Theorem I-2]{dmlr04}), when the isotropy group is the unitary group of a $*$-subalgebra of $M_n(\C)$. Our main result is the following.

\begin{thm}\label{main}
Let $\nn{\cdot}$ be a unitarily invariant norm on $\uu(n)$ with $\nn{iP}=1$ for rank-one projections $P$, let $\A\subset M_n(\C)$ be a unital $*$-subalgebra and let $M=\U(n)/\U(\A)$ with the quotient Finsler metric. If $X\in T_oM$, $\nn{X}_o<\frac{\pi}{2}$ and $z_0\in\uu(n)$ is a minimal lifting of $X$, then the curve $\delta(t)=e^{tz_0}\cdot o$, $t\in[0,1]$, is minimal: $L(\gamma)\ge L(\delta)=\nn{X}_o$ for every piecewise smooth curve $\gamma$ in $M$ joining $o$ to $e^{z_0}\cdot o$. In particular
$$
d_M(o,e^{z_0}\cdot o)=\nn{X}_o=\nn{z_0}.
$$
\end{thm}

In Section \ref{sun} we show that the same statement holds in the special unitary group, for every $\Ad$-invariant Finsler norm on $\su(n)$ (Theorem \ref{mainsu}); the main difference is that one has to control the trace of the logarithms involved. We also show that the convexity of the norm along segments holds for every $\Ad$-invariant Finsler norm on $\uu(n)$ (Proposition \ref{convadun}).

We remark that the hypothesis is on the tangent vector, and not on the curves we compare $\delta$ with (see Remark \ref{competitors}). In Section \ref{prelim} we recall the facts on unitarily invariant norms, rectifiable distances and quotient metrics that we need, and we prove the convexity of the spectral ball. Section \ref{proof} contains the proof of Theorem \ref{main}, Section \ref{sun} deals with $\SU(n)$, and in Section \ref{remarks} we discuss particular cases, the radius, and uniqueness.

\section{Preliminaries}\label{prelim}

\subsection{Unitarily invariant norms and the distance in $\U(n)$}

We will denote with $\uu(n)$ the space of skew-Hermitian $n\times n$ matrices, the Lie algebra of $\U(n)$. A norm $\nn{\cdot}$ on $\uu(n)$ is \emph{unitarily invariant} if $\nn{uxv}=\nn{x}$ for all $u,v\in\U(n)$ such that $uxv\in\uu(n)$; equivalently, $\nn{x}=\Phi(s(x))$ for a symmetric gauge function $\Phi$, where $s(x)$ is the vector of singular values of $x$. Such a norm is $\Ad$-invariant and symmetric ($\nn{-x}=\nn{x}$). Since symmetric gauge functions are monotone, $\Phi(s(x))\ge s_1(x)\Phi(1,0,\dots,0)$, that is $\nn{iP}\,\nn{x}_\infty\le\nn{x}$, where $\nn{\cdot}_\infty$ is the spectral norm and $P$ is a rank-one projection. Every unitarily invariant norm is a positive multiple of one with $\nn{iP}=1$, and we assume this normalization throughout; thus
\begin{equation}\label{comparison}
\nn{x}_\infty\le\nn{x}\qquad(x\in\uu(n)).
\end{equation}

The length of a piecewise smooth curve $\gamma\subset\U(n)$ is $L(\gamma)=\int\nn{\gamma^{-1}\dot\gamma}$, and $d$ is the induced rectifiable distance, which is bi-invariant. We write $\log$ for the principal branch of the logarithm, defined on unitaries that do not have $-1$ as an eigenvalue; note that $\nn{\log u}_\infty<\pi$ for such $u$.

\begin{thm}[{\cite[Theorem 16]{alv14}}]\label{alv}
If $z\in\uu(n)$ and $\nn{z}_\infty\le\pi$, then $t\mapsto ue^{tz}$, $t\in[0,1]$, has minimal length among the piecewise smooth curves joining $u$ to $ue^{z}$. In particular $d(u,ue^{z})=\nn{z}$.
\end{thm}

We denote with $d_\infty$ the rectifiable distance of the spectral norm. Every unitary has a logarithm with spectral norm at most $\pi$; hence $d_\infty(1,u)=\nn{\log u}_\infty$ if $-1$ is not an eigenvalue of $u$, and $d_\infty(1,u)=\pi$ otherwise.

\begin{thm}[{\cite[Theorems 3.9 and 4.2]{lm26}}]\label{lm}
Let $x,y\in\uu(n)$ and $\mu(s)=e^{sx}e^{y}$, $s\in[s_1,s_2]$, and assume that $\nn{\log\mu(s)}_\infty<\frac{\pi}{2}$ for all $s\in[s_1,s_2]$. Let $\theta_1(s)\ge\dots\ge\theta_n(s)$ be the arguments in $(-\pi,\pi)$ of the eigenvalues of $\mu(s)$. Then, for every $1\le m\le n$, the partial sum $\theta_1+\dots+\theta_m$ is convex and $\theta_m+\dots+\theta_n$ is concave; in particular $\theta_1$ is convex and $\theta_n$ is concave. Moreover, $s\mapsto\nn{\log\mu(s)}$ is convex for every unitarily invariant norm.
\end{thm}

In \cite{lm26} these statements are written for $e^{itx}e^{iy}$ with $x,y$ Hermitian, and with the hypothesis $\nn{\mu(s)-1}_\infty<\sqrt2$; this is the same, since $|e^{i\theta}-1|=2|\sin(\theta/2)|$. Note that $\nn{\log\mu(s)}_\infty=\max\{\theta_1(s),-\theta_n(s)\}$. Let us explain why $\frac{\pi}{2}$ is the natural threshold. If the eigenvalues are simple, then $\theta_k''=\sum_{j\ne k}\cot\frac{\theta_k-\theta_j}{2}\,|x_{kj}|^2$, where $x_{kj}$ are the entries of $-ix$ in an eigenbasis of $\mu(s)$ \cite[Remark 5.14]{lm26}. Summing over $k\le m$, the terms with $j\le m$ cancel in pairs, and the remaining coefficients are nonnegative as long as $0\le\theta_k-\theta_j<\pi$ for $k\le m<j$, which is the case if all the angles lie in $(-\frac{\pi}{2},\frac{\pi}{2})$.

\begin{lem}\label{ball}
Let $a,b\in\uu(n)$ with $\nn{a}_\infty,\nn{b}_\infty<\frac{\pi}{2}$, let $c=\log(e^{-a}e^{b})$ and let $\nu(s)=e^{a}e^{sc}$, $s\in[0,1]$, be the short geodesic joining $e^{a}$ to $e^{b}$. Then $s\mapsto\nn{\log\nu(s)}_\infty$ is convex in $[0,1]$, and
$$
\nn{\log\nu(s)}_\infty\le\max\{\nn{a}_\infty,\nn{b}_\infty\}<\tfrac{\pi}{2}\qquad(s\in[0,1]).
$$
In other words, the spectral ball of radius $\frac{\pi}{2}$ around $1$ is convex along the segments $s\mapsto e^ae^{sc}$ with $\nn{c}_\infty<\pi$.
\end{lem}
\begin{proof}
Since $d_\infty(1,e^{-a}e^{b})=d_\infty(e^a,e^b)\le\nn{a}_\infty+\nn{b}_\infty<\pi$, the logarithm $c$ is well defined and $\nn{c}_\infty<\pi$. For $\lambda\in[0,1]$ let $c_\lambda=\log(e^{-\lambda a}e^{\lambda b})$, $\nu_\lambda(s)=e^{\lambda a}e^{sc_\lambda}$ and $g_\lambda(s)=\nn{\log\nu_\lambda(s)}_\infty$. By the same estimate $\lambda\mapsto c_\lambda$ is continuous, hence $(\lambda,s)\mapsto g_\lambda(s)$ is continuous where $-1$ is not an eigenvalue of $\nu_\lambda(s)$. Let
$$
\Lambda=\{\lambda\in[0,1]: g_\lambda(s)<\tfrac{\pi}{2}\ \text{for all}\ s\in[0,1]\}.
$$
Clearly $0\in\Lambda$, and $\Lambda$ is open by continuity and compactness of $[0,1]$. We claim that $\Lambda$ is closed. If $\lambda\in\Lambda$, note that $\nu_\lambda(s)=e^{\lambda a}\mu(s)e^{-\lambda a}$ with $\mu(s)=e^{sc_\lambda}e^{\lambda a}$, so by Theorem \ref{lm} $g_\lambda=\max\{\theta_1,-\theta_n\}$ is convex. Let $\lambda_k\in\Lambda$, $\lambda_k\to\lambda$. The eigenvalues of $\nu_{\lambda_k}(s)$ have arguments in $(-\frac{\pi}{2},\frac{\pi}{2})$, therefore by continuity of the spectrum those of $\nu_\lambda(s)$ have arguments in $[-\frac{\pi}{2},\frac{\pi}{2}]$. In particular $-1$ is not an eigenvalue of $\nu_\lambda(s)$, and $g_\lambda$ is the pointwise limit of the convex functions $g_{\lambda_k}$, hence convex. Since $\nu_\lambda(0)=e^{\lambda a}$, $\nu_\lambda(1)=e^{\lambda b}$ and $\lambda a,\lambda b$ have spectral norm less than $\pi$, we get
$$
g_\lambda(s)\le\max\{g_\lambda(0),g_\lambda(1)\}=\lambda\max\{\nn{a}_\infty,\nn{b}_\infty\}<\tfrac{\pi}{2},
$$
which proves the claim. By connectedness $\Lambda=[0,1]$, and the case $\lambda=1$ is the assertion of the lemma.
\end{proof}

\begin{rem}
We stress that the lemma concerns only these one-parameter segments: short geodesics for the spectral norm are far from unique. The convexity of the distance along segments was established in \cite[Section 6]{lm26}, but the invariance of the spectral ball under these segments was not considered there; it is precisely what we need below. For $\SU(2)\simeq S^3$, the lemma is the convexity of an open hemisphere, and the same example shows that the radius $\frac{\pi}{2}$ cannot be enlarged. In \cite[Corollary 3.8]{alr10} and \cite[Corollary 2.9]{al10} the radius of convexity $\frac{\pi}{4}$ was obtained, by means of the triangle inequality, for the balls of the distances $d_p$ and $d_\infty$; for the spectral norm, Lemma \ref{ball} improves it to $\frac{\pi}{2}$.
\end{rem}

\subsection{Norming functionals}

Let $z\in\uu(n)$. A real linear functional $\psi$ on $\uu(n)$ is a \emph{norming functional} of $z$ if $\psi(z)=\nn{z}$ and $\psi(v)\le\nn{v}$ for all $v\in\uu(n)$. The set $\partial\nn{z}$ of norming functionals is the subdifferential of the norm at $z$, and the directional derivative
$$
\nn{\cdot}'(z;v)=\lim_{t\to0^+}\frac{\nn{z+tv}-\nn{z}}{t}=\max\{\psi(v):\psi\in\partial\nn{z}\}
$$
exists for every $v$ \cite[Theorem 23.4]{rockafellar}.

\begin{lem}\label{bracket}
If $\psi$ is a norming functional of $z$, then $\psi([w,z])=0$ for every $w\in\uu(n)$. Consequently $\psi((\ad z)^kv)=0$ for all $v\in\uu(n)$ and $k\ge1$.
\end{lem}
\begin{proof}
The function $h(t)=\psi(\Ad_{e^{tw}}z)$ satisfies $h(t)\le\nn{\Ad_{e^{tw}}z}=\nn{z}=h(0)$, therefore $h'(0)=\psi([w,z])=0$. For the second assertion, note that $(\ad z)^kv=-[(\ad z)^{k-1}v,z]$.
\end{proof}

This is the finite-dimensional case of \cite[Proposition 4.3]{lar19}; for smooth, strongly convex norms it is the criterion $g_z(z,[z,w])=0$ of \cite[Theorem 3.2]{lat07}.

\begin{lem}\label{derivlog}
Let $z,y\in\uu(n)$ with $\nn{z}_\infty<\pi$, and let $\ell(s)=\log(e^ze^{sy})$ for $|s|$ small. Then $\ell$ is smooth, $\ell(0)=z$, and
$$
\ell'(0)=g(\ad z)\,y,\qquad g(\lambda)=\frac{\lambda}{1-e^{-\lambda}}.
$$
Moreover, $\psi(\ell'(0))=\psi(y)$ for every norming functional $\psi$ of $z$.
\end{lem}
\begin{proof}
The principal logarithm is analytic near $e^z$, so $\ell$ is smooth and $\ell(0)=z$. Differentiating $e^{\ell(s)}=e^ze^{sy}$ at $s=0$, and recalling that $\exp_{*z}(v)=e^z\frac{1-e^{-\ad z}}{\ad z}(v)$, we obtain $\frac{1-e^{-\ad z}}{\ad z}(\ell'(0))=y$. The eigenvalues of $\ad z$ are $i(\theta_j-\theta_k)$, where $i\theta_j$ are those of $z$, and $|\theta_j-\theta_k|\le2\nn{z}_\infty<2\pi$; since $g$ is analytic in the disk $|\lambda|<2\pi$, the operator $\frac{1-e^{-\ad z}}{\ad z}$ is invertible, with inverse $g(\ad z)$. Finally $g(\lambda)=1+\sum_{k\ge1}c_k\lambda^k$ with $c_k\in\R$, so $g(\ad z)$ preserves $\uu(n)$ and Lemma \ref{bracket} gives $\psi(g(\ad z)y)=\psi(y)$.
\end{proof}

\subsection{Quotient metrics}

Let $\A\subset M_n(\C)$ be a unital $*$-subalgebra and $K=\U(\A)=\U(n)\cap\A$, a closed subgroup of $\U(n)$ with Lie algebra $\mathfrak k=\{v\in\A: v^*=-v\}$. Let $M=\U(n)/K$, with base point $o=K$, the action $g\cdot(hK)=ghK$ and the projection $q(g)=g\cdot o$. The quotient Finsler metric is given at $o$ by
$$
\nn{X}_o=\inf\{\nn{z}: z\in\uu(n),\ q_{*1}(z)=X\}=\dist(z,\mathfrak k),
$$
where $z$ is any lifting of $X$, and it is transported to $M$ by the action. It is well defined since the norm is $\Ad$-invariant and $\Ad_k\mathfrak k=\mathfrak k$ for $k\in K$. Since $\mathfrak k$ is finite dimensional, every $X$ has a \emph{minimal lifting} $z_0$, i.e. $\nn{z_0}=\nn{X}_o$, or equivalently
\begin{equation}\label{minlift}
\nn{z_0+v}\ge\nn{z_0}\qquad\text{for all } v\in\mathfrak k.
\end{equation}
Lengths of curves in $M$ and the rectifiable distance $d_M$ are defined as usual. The following is \cite[Theorem 3.26]{lar19}.

\begin{prop}\label{quotdist}
For $g\in\U(n)$, $d_M(o,g\cdot o)=\inf\{d(1,gk): k\in K\}$.
\end{prop}

Let us check that the hypotheses of \cite[Theorem 3.26]{lar19} hold. The projection $q$ is a smooth submersion, hence it has the path-lifting property of \cite[Definition 3.1]{lar19} (see \cite[Remark 3.3]{lar19}); the metric of $\U(n)$ is left-invariant, hence $L$-uniform with $L\equiv1$ \cite[Remark 2.5]{lar19}, and it is also right-invariant, which is used in the construction of almost isometric lifts \cite[Theorem 3.25]{lar19}. Finally, in \cite{lar19} the isotropy group is taken to be the identity component of the stabilizer; this is not needed for the inequality $d_M(o,g\cdot o)\ge\inf_{k\in K}d(1,gk)$, which is the one we use: the proof of \cite[Theorem 3.24]{lar19} only uses that the endpoint of a lift of a curve of $M$ lies in the fibre $gK$, with $K$ the full stabilizer.

\begin{ex}
Let $h$ be a Hermitian matrix and $\A=\{h\}'$ its commutant. Then $K$ is the isotropy group of $h$ for the action $u\cdot h=uhu^*$, $M$ is the unitary orbit $\mathcal O_h=\{uhu^*:u\in\U(n)\}$, and $q_{*1}(z)=[z,h]$. These are the flag manifolds; when $h$ is a projection, $\mathcal O_h$ is a Grassmann manifold.
\end{ex}

\section{Proof of Theorem \ref{main}}\label{proof}

The curve $\delta(t)=e^{tz_0}\cdot o$ has velocity $\dot\delta(t)=e^{tz_0}\cdot X$, therefore its length is $\nn{X}_o=\nn{z_0}$. By Proposition \ref{quotdist}, it suffices to prove that
\begin{equation}\label{goal}
d(1,w)\ge\nn{z_0}\qquad\text{for every } w\in e^{z_0}K.
\end{equation}
Suppose that $w\in e^{z_0}K$ and $d(1,w)<\nn{z_0}$; we will arrive at a contradiction.

By \eqref{comparison}, $\nn{z_0}_\infty\le\nn{z_0}<\frac{\pi}{2}$. Let $Z$ be a logarithm of $w$ with $\nn{Z}_\infty\le\pi$. By Theorem \ref{alv}, $\nn{Z}=d(1,w)<\nn{z_0}$, hence $\nn{Z}_\infty<\frac{\pi}{2}$; in particular $Z=\log w$.

Let $k=e^{-z_0}w\in K$. By Lemma \ref{ball} with $a=z_0$ and $b=\log w$, the element $y=\log k$ is well defined, and the short geodesic $\nu(s)=e^{z_0}e^{sy}$, $s\in[0,1]$, joining $e^{z_0}$ to $w$, satisfies $\nn{\log\nu(s)}_\infty<\frac{\pi}{2}$. Note that $y$ is a polynomial in $k$ (interpolate $\log$ on the spectrum of $k$), hence $y\in\A$; since $y^*=-y$, we have $y\in\mathfrak k$, and $\nu$ lies in the fibre $e^{z_0}K$.

Let $f(s)=\nn{\log\nu(s)}$. Since $\nu(s)=e^{z_0}\mu(s)e^{-z_0}$ with $\mu(s)=e^{sy}e^{z_0}$ and the norm is unitarily invariant, $f(s)=\nn{\log\mu(s)}$ and $\nn{\log\mu(s)}_\infty<\frac{\pi}{2}$. By Theorem \ref{lm}, $f$ is convex in $[0,1]$.

It remains to compute the right derivative of $f$ at $s=0$. Let $\ell(s)=\log\nu(s)$. By Lemma \ref{derivlog}, $\ell(s)=z_0+s\ell'(0)+o(s)$, and since the norm is Lipschitz,
$$
f'(0^+)=\nn{\cdot}'(z_0;\ell'(0))=\max_{\psi\in\partial\nn{z_0}}\psi(\ell'(0))=\max_{\psi\in\partial\nn{z_0}}\psi(y)=\nn{\cdot}'(z_0;y).
$$
Since $y\in\mathfrak k$, by \eqref{minlift} we have $\nn{z_0+ty}\ge\nn{z_0}$ for $t>0$, hence $f'(0^+)\ge0$.

A convex function on $[0,1]$ with nonnegative right derivative at $0$ is nondecreasing. Therefore $d(1,w)=\nn{\log w}=f(1)\ge f(0)=\nn{z_0}$, a contradiction. This proves \eqref{goal}, and the theorem follows. \hfill$\square$

\section{The special unitary group}\label{sun}

Let $\nn{\cdot}$ be an $\Ad$-invariant Finsler norm on $\su(n)$, that is, a subadditive, nondegenerate and positively homogeneous function ($\nn{\lambda v}=\lambda\nn{v}$ for $\lambda\ge0$) with $\nn{\Ad_uv}=\nn{v}$ for $u\in\SU(n)$; we do not assume $\nn{-v}=\nn{v}$. Lengths of curves in $\SU(n)$ and the rectifiable distance $d$ are defined as in Section \ref{prelim}; the distance is bi-invariant, and it is symmetric if and only if the norm is. Throughout, $d(u,w)$ is the distance \emph{from} $u$ \emph{to} $w$: the infimum of $L(\gamma)=\int\nn{\gamma^{-1}\dot\gamma}$ over curves $\gamma$ starting at $u$ and ending at $w$. By compactness and nondegeneracy, $\min\{\nn{x}: x\in\su(n),\ \nn{x}_\infty=1\}>0$, and we normalize the norm so that this minimum is $1$; thus
\begin{equation}\label{comparisonsu}
\nn{x}_\infty\le\nn{x}\qquad(x\in\su(n)).
\end{equation}
Only the inequality \eqref{comparisonsu} is used below. For a unitarily invariant norm with $\nn{iP}=1$ it holds on $\su(n)$ by \eqref{comparison}, although the minimum above may be larger than $1$, since $iP\notin\su(n)$. Let $\A\subset M_n(\C)$ be a unital $*$-subalgebra, $K=\SU(n)\cap\A$, with Lie algebra $\mathfrak k=\su(n)\cap\A$, and $M=\SU(n)/K$ with the quotient Finsler metric, defined as in Section \ref{prelim}. Note that $K$ need not be connected (for instance, $\A=\C1$ gives $K=\{\omega1:\omega^n=1\}$).

\begin{thm}\label{mainsu}
With the notations above, if $X\in T_oM$, $\nn{X}_o<\frac{\pi}{2}$ and $z_0\in\su(n)$ is a minimal lifting of $X$, then $\delta(t)=e^{tz_0}\cdot o$, $t\in[0,1]$, is minimal: $L(\gamma)\ge L(\delta)=\nn{X}_o$ for every piecewise smooth curve $\gamma$ in $M$ from $o$ to $e^{z_0}\cdot o$, and $d_M(o,e^{z_0}\cdot o)=\nn{X}_o=\nn{z_0}$. Here the norm is normalized as in \eqref{comparisonsu}, and $d_M(o,\cdot)$ stands for the distance \emph{from} $o$; since the norm is not symmetric, we do not claim anything about the distance \emph{to} $o$.
\end{thm}

The proof follows the lines of Section \ref{proof}; let us list the ingredients and the changes. Theorem \ref{alv} is replaced by \cite[Theorem 5.19(3) and Corollary 5.16]{lm23}: if $z\in\su(n)$ and $\nn{z}_\infty<\pi$, then $d(u,ue^z)=\nn{z}$, hence $ue^{tz}$, $t\in[0,1]$, is short; this holds for norms which are only positively homogeneous (we only use it for $\nn{z}_\infty<\frac{\pi}{2}$). The convexity in Theorem \ref{lm} is replaced by \cite[Theorems 5.8 and 5.9]{lm26}: if $\mu(s)=e^{sx}e^{y}=e^{x(s)}$ with $x(s)\in\su(n)$ and $\nn{x(s)}_\infty<\frac{\pi}{2}$ on an interval, then $s\mapsto\nn{x(s)}$ is convex there. The hypothesis $x(s)\in\su(n)$ is essential, and it is the reason for Lemma \ref{trace} below. Lemmas \ref{bracket} and \ref{derivlog} only use $\Ad$-invariance and sublinearity of the norm, and $\ad z$ preserves $\su(n)$; they hold verbatim (recall that \cite[Theorem 23.4]{rockafellar} holds for finite convex functions, therefore for non-symmetric norms). Proposition \ref{quotdist} holds by \cite[Theorem 3.26]{lar19}, whose framework allows non-symmetric norms (\cite[Definition 2.2]{lar19}), and as remarked in Section \ref{prelim} the connectedness of $K$ plays no role.

The new point is that the principal logarithm of an element of $\SU(n)$ is not necessarily traceless (consider $e^{2\pi i/5}1\in\SU(5)$, whose principal logarithm has spectral norm $\frac{2\pi}{5}<\frac{\pi}{2}$ and trace $2\pi i$; its spectral distance to $1$ in $\U(5)$ is $\frac{2\pi}{5}$, but its logarithms in $\su(5)$ have spectral norm at least $\frac{8\pi}{5}$, so by Lemma \ref{trace}(1) below its distance to $1$ in $\SU(5)$ is at least $\frac{\pi}{2}$ for every normalized norm). The following observation takes care of this.

\begin{lem}\label{trace}
Let $\gamma:[0,1]\to\SU(n)$ be continuous with $\gamma(0)=1$, and such that $-1$ is not an eigenvalue of $\gamma(t)$ for $t\in[0,1]$. Then $\log\gamma(t)\in\su(n)$ for all $t$. In particular:
\begin{enumerate}
\item if $w\in\SU(n)$ and $d(1,w)<\frac{\pi}{2}$, then $\log w\in\su(n)$, $\nn{\log w}_\infty\le d(1,w)$ and $d(1,w)=\nn{\log w}$;
\item if $a,b\in\su(n)$ and $\nn{a}_\infty,\nn{b}_\infty<\frac{\pi}{2}$, then $\log(e^{-a}e^b)\in\su(n)$, and the geodesic $\nu$ of Lemma \ref{ball} satisfies $\log\nu(s)\in\su(n)$ for $s\in[0,1]$.
\end{enumerate}
\end{lem}
\begin{proof}
Since $\det e^{A}=e^{\tr A}$, we have $\tr\log\gamma(t)\in2\pi i\mathbb Z$; this is a continuous function of $t$ which vanishes at $t=0$, hence it is identically $0$. For (1), let $\gamma$ be a curve in $\SU(n)$ joining $1$ to $w$ with $L(\gamma)<\frac{\pi}{2}$. By \eqref{comparisonsu}, its length for the spectral norm is less than $\frac{\pi}{2}$, therefore $\nn{\log\gamma(t)}_\infty=d_\infty(1,\gamma(t))<\frac{\pi}{2}$ for all $t$ (here $d_\infty$ is the spectral distance of $\U(n)$), and the first assertion applies. Taking the infimum over such $\gamma$ gives $\nn{\log w}_\infty\le d(1,w)$, and \cite[Theorem 5.19(3)]{lm23} gives $d(1,w)=\nn{\log w}$. For (2), apply the first assertion to $\lambda\mapsto e^{-\lambda a}e^{\lambda b}$, and to the concatenation of $\lambda\mapsto e^{\lambda a}$, $\lambda\in[0,1]$, with $s\mapsto\nu(s)$; neither curve has $-1$ as an eigenvalue, by the estimates in the proof of Lemma \ref{ball} and since $\nn{\lambda a}_\infty<\frac{\pi}{2}$.
\end{proof}

\begin{proof}[Proof of Theorem \ref{mainsu}]
Suppose that $w\in e^{z_0}K$ and $d(1,w)<\nn{z_0}<\frac{\pi}{2}$. By \eqref{comparisonsu}, $\nn{z_0}_\infty<\frac{\pi}{2}$, and by Lemma \ref{trace}(1), $\log w\in\su(n)$, $\nn{\log w}_\infty<\frac{\pi}{2}$ and $d(1,w)=\nn{\log w}$. Let $k=e^{-z_0}w\in K$ and $y=\log k$. By Lemma \ref{trace}(2), $y\in\su(n)$, and since $y$ is a polynomial in $k$, $y\in\A$; hence $y\in\mathfrak k$. By Lemmas \ref{ball} and \ref{trace}(2), $\nu(s)=e^{z_0}e^{sy}$ satisfies $\log\nu(s)\in\su(n)$ and $\nn{\log\nu(s)}_\infty<\frac{\pi}{2}$ for $s\in[0,1]$. Writing $\nu(s)=e^{z_0}\mu(s)e^{-z_0}$ with $\mu(s)=e^{sy}e^{z_0}$, by $\Ad$-invariance and \cite[Theorem 5.9]{lm26} the function $f(s)=\nn{\log\nu(s)}=\nn{\log\mu(s)}$ is convex in $[0,1]$. The computation of $f'(0^+)\ge0$ is the same as in Section \ref{proof}, and therefore $d(1,w)=f(1)\ge f(0)=\nn{z_0}$, a contradiction.
\end{proof}

\subsection{$\Ad$-invariant norms on $\uu(n)$}\label{adun}

The convexity of the norm along segments does not need the norm to be unitarily invariant, nor the trace to vanish. Let $\nn{\cdot}$ be an $\Ad$-invariant Finsler norm on $\uu(n)$, not necessarily symmetric.

\begin{prop}\label{convadun}
Let $x,y\in\uu(n)$ and $\mu(s)=e^{sx}e^{y}$, and assume that $\nn{\log\mu(s)}_\infty<\frac{\pi}{2}$ for $s\in[s_1,s_2]$. Then $s\mapsto\nn{\log\mu(s)}$ is convex in $[s_1,s_2]$.
\end{prop}
\begin{proof}
Write $\log\mu(s)=ih(s)$ with $h(s)$ Hermitian, with eigenvalues $\theta_1(s)\ge\dots\ge\theta_n(s)$ (the arguments of Theorem \ref{lm}), and let $S_k=\theta_1+\dots+\theta_k$. Since the norm is sublinear, $\nn{ih}=\max\{\tr(ah):a\in B\}$, where $B$ is the (compact) set of Hermitian $a$ with $\tr(ah')\le\nn{ih'}$ for all Hermitian $h'$; by $\Ad$-invariance, $B$ is invariant under $a\mapsto uau^*$. For a fixed Hermitian $a$ with eigenvalues $\alpha_1\ge\dots\ge\alpha_n$, a classical inequality of von Neumann gives $\max_{u\in\U(n)}\tr(uau^*h)=\sum_k\alpha_k\theta_k$, where $\theta_k$ are the eigenvalues of $h$ in nonincreasing order (see \cite[Chapter III]{bhatia}). Hence
$$
\nn{\log\mu(s)}=\max_{a\in B}\ \sum_{k=1}^{n-1}(\alpha_k-\alpha_{k+1})S_k(s)+\alpha_nS_n(s).
$$
By Theorem \ref{lm}, $S_k$ is convex for $k<n$, and the coefficients $\alpha_k-\alpha_{k+1}$ are nonnegative. On the other hand $iS_n(s)=\tr\log\mu(s)$ and $e^{\tr\log\mu(s)}=\det\mu(s)=e^{s\tr x+\tr y}$; therefore $\tr\log\mu(s)-s\tr x-\tr y$ is a continuous function with values in $2\pi i\mathbb Z$, hence constant, and $S_n$ is affine. Thus each term is convex, whatever the sign of $\alpha_n$, and so is their maximum.
\end{proof}

For $\su(n)$ this gives a short proof of the convexity in \cite[Theorems 5.8 and 5.9]{lm26}, without Kostant's convexity theorem. To obtain Theorem \ref{main} in $\U(n)/\U(\A)$ for these norms (normalized so that $\nn{x}_\infty\le\nn{x}$), the proof of Section \ref{proof} goes through with Proposition \ref{convadun} in place of Theorem \ref{lm}, provided Theorem \ref{alv} holds, at least for $\nn{z}_\infty<\frac{\pi}{2}$. For unitarily invariant norms this is \cite{alv14}, and for $\su(n)$ it is \cite{lm23}; but \cite{lm23} deals only with compact groups with semisimple Lie algebra, and for general $\Ad$-invariant norms on $\uu(n)$ we know it only near the identity \cite[Theorem 4.11]{lar19}. It would be natural to conjecture that $d(1,e^z)=\nn{z}$ whenever $\nn{z}_\infty\le\pi$.

\section{Remarks}\label{remarks}

\begin{rem}[Particular cases]
For the Frobenius norm, the quotient metric is the normal homogeneous Riemannian metric of $M$. For the $p$-Schatten norms with $p$ even, Theorem \ref{main} recovers \cite[Theorem 4.9]{alr10} in finite dimension and for these isotropy groups, with the radius $\frac{\pi}{4}$ improved to $\frac{\pi}{2}$. For the operator norm, Theorem \ref{main} is \cite[Theorem I-2]{dmlr04}; our proof is different. The new cases are the $p$-Schatten norms for $p$ odd, the Ky Fan norms and, in general, every unitarily invariant norm other than the operator norm. Note also that, by \eqref{goal}, the infimum in Proposition \ref{quotdist} is attained at $k=1$: the distance in $M$ is computed by the minimal lifting.
\end{rem}

\begin{rem}[The competitors]\label{competitors}
The smallness assumption is on $X$, and the curves we compare $\delta$ with are arbitrary. The proof only uses the endpoint $w$ of a lift of the competitor (Proposition \ref{quotdist}); if the competitor is shorter than $\delta$, then $\nn{\log w}_\infty<\frac{\pi}{2}$, and convexity is applied, not to the competitor, but to the segment $\nu$ of the fibre, which stays in the spectral ball of radius $\frac{\pi}{2}$ by Lemma \ref{ball}. Since $d_M(o,e^{z_0}\cdot o)=\nn{X}_o$, the curve $\delta$ is also minimal among rectifiable curves, measured with their metric length: the metric length of a curve is never smaller than the distance between its endpoints.
\end{rem}

\begin{rem}[The radius]
The hypothesis $\nn{X}_o<\frac{\pi}{2}$ guarantees that $w$ and, by Lemma \ref{ball}, the whole segment $\nu$ lie in the region where Theorem \ref{lm} holds, and that region is optimal \cite{lm26}. For Grassmann manifolds the same radius appears in \cite{pr87} for the operator norm, and in \cite[Theorem 20]{alv14} for unitarily invariant norms. For the operator norm, the minimality of $\delta$ with $\nn{X}_o<\frac{\pi}{2}$ is \cite[Theorem I-2]{dmlr04}. A second proof, for orbits of finite-rank operators, was given in \cite[Theorem 5.3]{al10} by means of the convexity of the distance along the segment of the fibre; this requires the segment to stay in the spectral ball of radius $\frac{\pi}{2}$, and Lemma \ref{ball} makes this step explicit: when the generators have finite rank, all the unitaries involved lie in $1+B(F)$ for a finite-dimensional subspace $F$, and the problem reduces to $\U(F)$. The same discussion applies to Theorem \ref{mainsu}, with \cite[Theorems 5.8 and 5.9]{lm26} in place of Theorem \ref{lm}; the optimality of the region of convexity in $\SU(n)$ is \cite[Remark 6.4]{lm26}, by the same example $\SU(2)\simeq S^3$. The radius cannot be improved in general: for the spectral norm on $\mathbb{CP}^1\simeq S^2$ (the Grassmann manifold of lines in $\C^2$) one has $\nn{X}_o=\theta$, the principal angle, and $\delta$ stops being short at $\theta=\frac{\pi}{2}$. For a given norm the optimal radius may be larger; for instance, for the Frobenius norm on $\mathbb{CP}^1$ the minimal lifting has singular values $(\theta,\theta)$, and $\delta$ is short up to $\nn{X}_o=\pi/\sqrt{2}$.
\end{rem}

\begin{rem}[Uniqueness]
If $s\mapsto\nn{\log\nu(s)}$ is strictly convex for every $w\ne e^{z_0}$ in the fibre with $d(1,w)\le\nn{z_0}$, the same argument shows that $e^{z_0}$ is the unique point of $e^{z_0}K$ closest to $1$. This is a statement about endpoints of lifts, not about curves: a competitor of length $\nn{z_0}$ only has lifts of length at most $\nn{z_0}+\varepsilon$, and uniqueness of $\delta$ would require an additional argument, for instance the existence of isometric lifts. A word of caution: strict convexity may fail along the fibre, since if $y$ commutes with $z_0$, then $\nn{\log\nu(s)}=\nn{z_0+sy}$. By \cite[Theorem 5.13]{lm26}, it holds for orbit norms with distinct weights when the generators do not commute. We do not pursue the question of uniqueness here.
\end{rem}

\begin{rem}[Other settings]
The argument only needs: (i) convexity of the distance from the identity along geodesic segments, in a neighbourhood containing the relevant segments of the fibre; (ii) logarithms of elements of the isotropy group, near the identity, in its Lie algebra; and (iii) Proposition \ref{quotdist}. In $\U(n)$, (i) is Theorem \ref{lm}; in $\SU(n)$ it is \cite[Theorems 5.8 and 5.9]{lm26}, and (ii) needs in addition the control of the trace of Lemma \ref{trace}. The extension to the unitary groups of operator ideals would require an infinite-dimensional version of Theorem \ref{lm}.
\end{rem}

\section*{Acknowledgements}
The author thanks Mart\'in Miglioli for many conversations on these topics. This research was partially supported by CONICET, ANPCyT and Universidad de Buenos Aires.

The author used an AI assistant (Claude, Anthropic) to help draft and edit the text, to check the references against the original sources, and to suggest arguments for some auxiliary results (e.g. Proposition \ref{convadun}), all of which were verified by the author. The author takes full responsibility for the content of this paper.

\end{document}